\documentclass{amsart}
\usepackage[utf8]{inputenc}
\usepackage{amsfonts}
\usepackage{amsmath}
\usepackage{amsthm}
\usepackage{float}
\usepackage{graphicx}
\usepackage{subfigure}
\usepackage[pagewise]{lineno}\nolinenumbers

\usepackage{witharrows}
\usepackage{mathtools}

\newtheorem{thm}{Theorem}
\newtheorem{prp}[thm]{Proposition} 
\newtheorem{ex}[thm]{Example}
\newtheorem{rem}[thm]{Remark}
\newtheorem{lem}[thm]{Lemma}
\newtheorem{defn}[thm]{Definition}
\newtheorem{cor}[thm]{Corollary}

\newtheorem*{thmLL}{Theorem (Li--Lu)}
\newtheorem*{thmA}{Theorem A} 
\newtheorem*{thmB}{Theorem B}

\title[Periodicity and rotation number]{Periodicity and rotation number for random circle homeomorphisms}
\author{Zixu Li}
\address{School of Mathematics, Sun Yat-sen University, Zhuhai, China}
\email{zixu.li915@gmail.com}
\thanks{This work was partially supported by fund PGRS2006028.}
\author{Simon Lloyd}
\address{School of Mathematics and Physics, Xi'an Jiaotong-Liverpool University, China}
\email{Simon.Lloyd@xjtlu.edu.cn}
\date{\today}
 % Rotation numbers and vectors

\begin{document}

\begin{abstract}
We study discrete-time random dynamical systems where each fibre map is an orientation-preserving homeomorphism of the circle. We prove that the existence of a random periodic cycle with period at least two implies that the random rotation number is rational almost surely. Moreover, in clear contrast with the deterministic setting, we demonstrate that a common fixed point for the fibre maps does not imply that the random rotation number is an integer. Conversely, we show that if the mean random rotation number is an integer, then the fibre maps have a fixed point with positive probability.
\end{abstract}

\maketitle

\section{Introduction}

The notion of rotation number is an important topological invariant in the study of the dynamics of orientation-preserving circle homeomorphisms. Introduced by Poincar\'e, it is a measure of the average amount of rotation per iteration along an orbit. A key property of the  rotation number is that it can be used to characterise the existence of periodic points: the rotation number is a rational number if and only if there is a periodic orbit. In particular, fixed points exist if and only if the rotation number is an integer. See \cite{deFariaGuarino2024}, \cite{KatokHasselblatt1995} or \cite{deMelovanStrien1993} for an overview of the classical theory.
In this article, we consider the connections between periodicity and rotation number in the context of discrete-time random dynamical systems on the circle, and highlight how the theory differs from the deterministic setting of the dynamics of a single homeomorphism.

Let $\mathbb{S}^1$ denote the circle, parametrised by $x\in [0,1)$, and let $\mathcal{H}$ denote the collection of orientation-preserving homeomorphisms of $\mathbb{S}^1$, with the uniform topology and Borel sigma-algebra. 
Given a probability space consisting of a set $\Omega$, a sigma-algebra $\mathcal{F}$ of measurable subsets and a probability measure $\mathbb{P}$, let 
$(\Omega,\mathcal{F},\mathbb{P},\sigma)$ denote the \emph{measure-preserving dynamical system} for which $\sigma:\Omega\to\Omega$ is a measurable transformation that preserves $\mathbb{P}$: that is, $\mathbb{P}(\sigma^{-1}(A))=\mathbb{P}(A)$ for each $A\in\mathcal{F}$. 
A \emph{random circle homeomorphism (RCH) over $(\Omega,\mathcal{F},\mathbb{P},\sigma)$} is a measurable map $\Phi:\mathbb{N}_0\times\Omega\times \mathbb{S}^1\to \mathbb{S}^1$ with the following properties for each $n,m\in \mathbb{N}_0=\{0,1,2,\ldots\}$, $\omega\in\Omega$ and $x\in \mathbb{S}^1$:
\begin{enumerate}
\item $\Phi(0,\omega,\cdot)=\mathrm{Id}_{\mathbb{S}^1}$;
\item $\Phi(n,\omega,\cdot)\in \mathcal{H}$;
\item $\Phi(m+n,\omega,x)=\Phi(n,\sigma^m\omega,\Phi(m,\omega,x))$.
\end{enumerate}
The measure-preserving dynamical system $(\Omega,\mathcal{F},\mathbb{P},\sigma)$ is called the \emph{base dynamics} of $\Phi$, and we write $\sigma\omega$ for $\sigma(\omega)$. If we denote $\Phi(n,\omega,\cdot)$ by $f^{(n)}_\omega$ and $\Phi(1,\omega,\cdot)$ by $f_\omega$, the \emph{fibre map} at $\omega\in\Omega$, then
\begin{align}\label{eqn:fnomega}
\Phi(n,\omega,\cdot)=f^{(n)}_\omega=f_{\sigma^{n-1}\omega}\circ \cdots \circ f_{\sigma\omega}\circ f_{\omega}.
\end{align}

Let $\mathcal{H}(\Omega)$ denote the collection of measurable maps $f:\Omega\to \mathcal{H}$, where $\omega\mapsto f_\omega$. We say $f\in \mathcal{H}(\Omega)$ is the \emph{generator} of the RCH $\Phi:\mathbb{N}_0\times\Omega\times \mathbb{S}^1\to \mathbb{S}^1$ over $(\Omega,\mathcal{F},\mathbb{P},\sigma)$ defined by (\ref{eqn:fnomega}). We shall refer to the RCH $\Phi$ by the pair $(\sigma,f)$ used to define it.

It is often convenient to lift the dynamics of circle maps to the real line. Let $\tilde{\mathcal{H}}$ denote the collection of continuous and strictly increasing functions $F:\mathbb{R}\to\mathbb{R}$ that have the \emph{degree 1 property}: that is, for all $x\in\mathbb{R}$,
\begin{align}\label{eqn:degree1}
F(x+1)=F(x)+1.
\end{align}
We endow $\tilde{\mathcal{H}}$ with the uniform topology and Borel sigma-algebra. Given a single orientation-preserving homeomorphism $f\in \mathcal{H}$ of the circle, a \emph{lift} is any function $F\in\tilde{\mathcal{H}}$ that satisfies $\pi\circ F=f\circ \pi$, where $\pi:\mathbb{R}\to\mathbb{S}^1$ is the projection map $\pi(x)=x-\lfloor x\rfloor=x\ \mathrm{mod}\:1$. Here $\lfloor\cdot\rfloor:\mathbb{R}\to\mathbb{Z}$ denotes the floor function, defined by $\lfloor x\rfloor=\max\{n\in\mathbb{Z}:n\leq x\}$. Any two such lifts differ by an integer. 

In a similar way, we can lift a RCH. Given the generator $f:\Omega\to \mathcal{H}$ of a RCH, we say a measurable function $F:\Omega\to \tilde{\mathcal{H}}$ is a \emph{random lift} of $f$ if $F_\omega$ is a lift of the fibre map $f_\omega$ for every $\omega\in\Omega$, and we denote the composition $F_{\sigma^{n-1}\omega}\circ\cdots\circ F_\omega$ by $F^{(n)}_\omega$. 
Denote by $\tilde{\mathcal{H}}(\Omega)$ the collection of measurable maps $F:\Omega\to \tilde{\mathcal{H}}$ that satisfy the integrability condition
\begin{align}\label{eqn:integrabilityDev}
\int_{\omega\in\Omega} \|F_\omega-\mathrm{Id}\|\:\mathrm{d}\mathbb{P}(\omega) < \infty,
\end{align}
where $\|\cdot\|$ denotes the uniform norm, and we identify elements of $\tilde{\mathcal{H}}(\Omega)$ that are equal $\mathbb{P}$-almost everywhere. 

The concept of rotation number was generalised by Herman \cite{Herman1983} to the setting of quasiperiodically-forced circle homeomorphisms: that is, a RCH for which the base dynamics is an irrational rotation of the circle and the dependence $\omega\mapsto f_\omega$ is continuous. The extension of rotation number to random circle homeomorphisms was made by Ruffino \cite{Ruffino2000} for random lifts $F$ satisfying $F_\omega(0)\in (-1/2,1/2]$ for all $\omega\in\Omega$. Using integrability condition (\ref{eqn:integrabilityDev}) for the random lifts, Li and Lu \cite{LiLu2008} demonstrated the existence of $\rho_F$, the \emph{random rotation number} of $F\in\tilde{\mathcal{H}}(\Omega)$, as an integrable function:

\begin{thmLL}
If $F\in \tilde{\mathcal{H}}(\Omega)$, then there exists $\rho_F\in L^1(\Omega)$ such that for each $x\in\mathbb{R}$
\begin{align}\label{eqn:randomrotationnumber}
\lim_{n\to\infty} \frac{F^{(n)}_\omega(x)-x}{n} = \rho_F(\omega)
\end{align}
for $\mathbb{P}$-almost every $\omega\in\Omega$. The function $\rho_F$ satisfies $\rho_F\circ \sigma=\rho_F$, and thus is essentially constant if $\sigma$ is ergodic.
Moreover, $\rho: \tilde{\mathcal{H}}(\Omega)\to L^1(\Omega)$ is continuous.
\end{thmLL}

We denote by $\rho(F)$ the \emph{mean random rotation number}, which is given by
\begin{align}
\rho(F)=\int_{\omega\in\Omega} \rho_F(\omega)\:\mathrm{d}\mathbb{P}(\omega).
\end{align}
In particular, if $\sigma$ is ergodic, then $\rho_F(\omega)=\rho(F)$ for $\mathbb{P}$-almost every $\omega\in\Omega$. Zmarrou and Homburg \cite{ZmarrouHomburg2008} studied certain one-parameter families of smooth RCHs, and provided conditions for the smooth dependence of $\rho(F)$ on the parameter.

Ruffino and Rodrigues \cite{RodriguesRuffino2013} investigated the dependence of the random rotation number on the choice of random lift. A notable choice is the random lift $F\in\tilde{\mathcal{H}}(\Omega)$ of $f\in \mathcal{H}(\Omega)$ that satisfies the condition
\begin{align}\label{eqn:stdlift}
F_\omega(0)\in [0,1), 
\end{align}
which is called \emph{the standard random lift}. We will work exclusively with the standard random lift throughout. In the same article, Ruffino and Rodrigues observed that, unlike in the deterministic setting, rationality of the random rotation number does not imply the existence of periodic solutions. 

The concepts of fixed points and periodic points have been adapted for random dynamical systems. Random fixed points have been studied by Arnold \cite{Arnold1995} and by Arnold and Schmalfuss \cite{ArnoldSchmalfuss1996}. In the context of random dynamical systems on the circle, a random fixed point is a measurable function $a:\Omega\to\mathbb{S}^1$ for which $f_\omega(a(\omega))=a(\sigma\omega)$ for $\mathbb{P}$-almost every $\omega\in\Omega$. This is analogous to the notion of an invariant graph for quasi-periodically forced systems (see, for example \cite{Keller1996}).
The concept of a random periodic cycle was introduced by Kl\"unger \cite{Klunger2001} in the context of random dynamical systems on the real line. For quasiperiodically-forced circle homeomorphisms, J\"ager and Keller \cite{JagerKeller2006} introduced $p,q$-invariant graphs (and related structures) in their development of a Denjoy theory for such systems.

The main results of this article describe relationships between the random rotation number of a RCH and periodicity in the dynamics. To the best of our knowledge, these are the first results that establish a link between these concepts in the general random setting. In Proposition \ref{prp:ranpqcycle} we show that for random circle homeomorphisms, every random periodic cycle is a random $(p,q)$-periodic cycle (see Definition \ref{defn:rpc}). The existence of a random periodic cycle of period at least two determines the random rotation number:

\begin{thmA}
Let $F\in\tilde{\mathcal{H}}(\Omega)$ be the standard random lift of the generator $f\in \mathcal{H}(\Omega)$ of a random circle homeomorphism over a measure-preserving dynamical system $(\Omega,\mathcal{F},\mathbb{P},\sigma)$. If there is a random $(p,q)$-periodic cycle for $(\sigma, f)$ with $1\leq p<q$, then for $\mathbb{P}$-almost every $\omega\in\Omega$, the random rotation number $\rho_F\in L^1(\Omega)$ satisfies
\begin{align}
\rho_F(\omega) = \frac{p}{q}.
\end{align}
\end{thmA}

The assumption that $1\leq p<q$ is necessary: if $p=0$, then the random periodic cycle reduces to a collection of random fixed points, the existence of which does not ensure that $\rho_F(\omega)$ is zero. Indeed, we demonstrate that a random composition of order-preserving homeomorphisms can have a non-zero random rotation number even when they share a common fixed point(see Example \ref{ex:fixedpoint}). 
In the quasiperiodically-forced setting, J\"ager and Keller \cite{JagerKeller2006} derived a related result showing that the existence of a continuous $p,q$-invariant graph forces a rational relationship between $\rho(F)$ and the rotation number of the base dynamics.

The ubiquity of deterministic periodic behaviour for randomly generated circle homeomorphisms has been investigated from several points of view.
Downarowicz et al.~\cite{Downarowiczetal1992} estimated the probability of fixed points for homeomorphisms generated by a canonical random scheme and raised the question of whether periodic points existed almost surely.
Using the Polish group structure of the space $\mathcal{\mathcal{H}}$, Darji et al.~\cite{Darjietal2020} showed that every non-Haar null conjugacy class consists of maps with finitely many periodic points and asked whether the set of maps without periodic points is Haar null.
Our final result shows that, for the standard random lift of a RCH, integer values of the mean random rotation number signify that deterministic fixed points exist with positive probability:

\begin{thmB}
Let $F\in\tilde{\mathcal{H}}(\Omega)$ be the standard random lift of the generator $f\in \mathcal{H}(\Omega)$ of a random circle homeomorphism over a measure-preserving dynamical system $(\Omega,\mathcal{F},\mathbb{P},\sigma)$. If the mean random rotation number satisfies $\rho(F)\in \mathbb{Z}$, then 
\begin{align}
\mathbb{P}(\{\omega\in\Omega: f_\omega\ \mathrm{has\ a\ fixed\ point}\})>0.
\end{align}
\end{thmB}

We demonstrate the set of maps with a fixed point in this result does not necessarily have full $\mathbb{P}$-measure (see Example \ref{ex:fixedpointhalf}). Finally, we give an example to show that Theorem B cannot be generalised to other values of $\rho(F)$, since non-integer values of the mean random rotation number do not ensure the existence of deterministic periodic points (Example \ref{ex:periodic0}).

\section{Random circle homeomorphisms}

Orientation-preserving homeomorphisms of the circle are homotopic to the identity, and it is helpful to work their deviations from the identity map.
Given a RCH $(\sigma,f)$ and an integrable random lift $F\in\tilde{\mathcal{H}}(\Omega)$, its \emph{deviation function} is the map $\mathrm{Dev}_F:\Omega\times \mathbb{R}\to\mathbb{R}$ given by
\begin{align}
\mathrm{Dev}_F(\omega,x)=F_\omega(x)-x.    
\end{align}
By (\ref{eqn:integrabilityDev}), we have $\mathrm{Dev}_F(\cdot, x)\in L^1(\Omega)$ for each $x\in\mathbb{R}$. Since $F_\omega$ satisfies (\ref{eqn:degree1}), the function $\mathrm{Dev}_F$ satisfies $\mathrm{Dev}_F(\omega,x+1)=\mathrm{Dev}_F(\omega,x)$ for all $\omega\in\Omega$ and $x\in\mathbb{R}$. Therefore, it is possible to project the deviation function of the random lift to the circle by defining $\mathrm{dev}_F:\Omega\times \mathbb{S}^1\to\mathbb{R}$ by $\mathrm{dev}_F(\omega,x)=\mathrm{Dev}_F(\omega,\pi^{-1}(x))$. This map can then be used to express the generator $f$ of the RCH as follows:
\begin{align}
f_\omega(x)=x+\mathrm{dev}_F(\omega,x)\ \mathrm{mod}\:1.    
\end{align}

The prototypical examples of orientation-preserving circle maps are the Arnold maps, which we generalise to the random setting. Given a pair of measurable functions $\beta:\Omega\to [0,1)$ and $\alpha:\Omega\to[-1,1]$, a \emph{random Arnold homeomorphism} over $\sigma$ is generated by the map $f\in \mathcal{H}(\Omega)$ defined by
\begin{align}\label{eqn:Arnold}
f_\omega(x) = x+\frac{\alpha(\omega)}{2\pi}\sin(2\pi x) + \beta(\omega) \ \mathrm{mod}\:1
\end{align}
for all $x\in \mathbb{S}^1$ and $\omega\in\Omega$. The requirement that $|\alpha|\leq 1$ ensures that each $f_\omega$ is injective.

The formalism of random circle homeomorphisms can be used to study the properties of long compositions of circle homeomorphisms that are chosen randomly and independently with respect to the same probability measure from a collection $\{g_0,g_1,\ldots,g_{q-1}\}\subset \mathcal{H}$. 
To achieve this, we can form a RCH where the base dynamics is a Bernoulli shift. More specifically, let $Z_q=\{0,1,\ldots, q-1\}$, set $\Omega=(Z_q)^\mathbb{Z}=\{\omega=(\omega_n)_{n\in\mathbb{Z}}:\omega_n\in Z_q,n\in\mathbb{Z}\}$ to be the collection of bi-infinite sequences on $q$ symbols with the product topology, and let $\sigma:\Omega\to\Omega$ denote the left shift transformation, where $(\sigma\omega)_n=\omega_{n+1}$ for each $n\in\mathbb{Z}$. Any choice of probability vector $(p_0,\ldots, p_{q-1})$, where $0< p_i< 1$ and $\sum_i p_i=1$, determines a probability measure $\mathbb{P}$ on $\Omega$. Namely,
for each \emph{cylinder set}
\begin{align}
C^{n_1,n_2,\ldots,n_k}_{\alpha_1,\alpha_2,\ldots,\alpha_k}=\{\omega\in\Omega: \omega_{n_i}=\alpha_i, 1\leq i\leq k\},
\end{align}
we fix
\begin{align}
\mathbb{P}\left(C^{n_1,n_2,\ldots,n_k}_{\alpha_1,\alpha_2,\ldots,\alpha_k}\right)=\prod_{i=1}^kp_{\alpha_i}
\end{align}
and then extend $\mathbb{P}$ to the Borel sigma-algebra $\mathcal{F}$. We call $(\Omega,\mathcal{F},\mathbb{P},\sigma)$ the \emph{two-sided $(p_0,\ldots,p_{q-1})$-Bernoulli shift} (see, for example, \cite{KatokHasselblatt1995}). Finally, we define the generator $f\in\mathcal{H}(\Omega)$ by setting $f_\omega=g_{\omega_0}$. Random walks on $\mathcal{H}$ formed in this way have been studied in connection with synchronisation phenomena (see Antonov \cite{Antonov1984}, Malicet \cite{Malicet2017}).

The notion of an invariant measure for a RCH $(\sigma,f)$ can be expressed in terms of the \emph{associated skew-product map} $S_f:\Omega\times\mathbb{S}^1\to\Omega\times\mathbb{S}^1$ given by
\begin{align}
(\omega,x)\mapsto (\sigma\omega, f_\omega(x)).
\end{align}
Denoting the Borel sigma-algebra of $\mathbb{S}^1$ by $\mathcal{B}(\mathbb{S}^1)$, we let $\mathcal{P}_{\mathbb{P}}(\Omega\times\mathbb{S}^1)$ denote the set of probability measures $\mu$ on the measurable space $(\Omega\times\mathbb{S}^1,\mathcal{F}\otimes\mathcal{B}(\mathbb{S}^1))$ with marginal $\mathbb{P}$: that is, measures $\mu$ for which $(\pi_{\Omega})_\star\mu=\mathbb{P}$, where $\pi_{\Omega}:\Omega\times\mathbb{S}^1\to\Omega$ is the projection onto the first coordinate.

Given the continuity of the fibre maps $f_\omega$ and the compactness of $\mathbb{S}^1$, it follows \cite[p.~31]{Arnold1998} that the subset of $\mathcal{P}_{\mathbb{P}}(\Omega\times\mathbb{S}^1)$ consisting of $S_f$-invariant probability measures is non-empty.
An invariant probability measure $\mu\in\mathcal{P}_{\mathbb{P}}(\Omega\times\mathbb{S}^1)$ can be disintegrated to give an essentially unique family $\{\mu_\omega\}_{\omega\in\Omega}$ of \emph{sample measures} on $\mathbb{S}^1$, such that $\omega\mapsto \mu_\omega(B)$ is measurable for each $B\in\mathcal{B}(\mathbb{S}^1)$ and for each $U\in \mathcal{F}\otimes\mathcal{B}(\mathbb{S}^1)$ we have
\begin{align}\label{eqn:samplemeasure}
\mu(U)= \int_{\omega\in\Omega}\mu_\omega(U_\omega)\:\mathrm{d}\mathbb{P}(\omega),
\end{align}
where $U_\omega$ denotes the $\omega$-section of $U$ (see Arnold \cite{Arnold1998}). As shown by Ruffino and Rodrigues \cite{RodriguesRuffino2013}, given an invariant probability measure $\mu\in\mathcal{P}_{\mathbb{P}}(\Omega\times\mathbb{S}^1)$, by the Birkhoff ergodic theorem applied to the skew product $S_f$, the sequence
\begin{align}\label{eqn:devsum}
\frac{1}{n}\left(F^{(n)}_\omega(x)-x\right) = \frac{1}{n}\sum_{i=0}^{n-1} \mathrm{dev}_F\left(S^i_f(\omega,x)\right)
=\frac{1}{n}\sum_{i=0}^{n-1} \mathrm{dev}_F\left(\sigma^i\omega,f^{(i)}_{\omega}(x)\right)
\end{align}
converges for $\mu$-almost every pair $(\omega,x)\in \Omega\times\mathbb{S}^1$ to the limit $\rho_F(\omega)$ that is independent of $x$, and so
\begin{align}\label{eqn:devint}
\rho(F)
=\int_\Omega\rho_F\:\mathrm{d}\mathbb{P}
=\int_{\omega\in\Omega}\int_{\mathbb{S}^1}\rho_F(\omega)\:\mathrm{d}\mu_\omega\:\mathrm{d}\mathbb{P}(\omega)
=\int_{\Omega\times\mathbb{S}^1}\mathrm{dev}_F\:\mathrm{d}\mu.
\end{align}

The basic inequalities satisfied by the standard random lift are given in the following lemma. 

\begin{lem}\label{lem:Fbound}
Let $F\in \tilde{\mathcal{H}}(\Omega)$ be the standard random lift of a random order-preserving circle map $f\in \mathcal{H}(\Omega)$. The random lift $F$ satisfies for all $x\in \mathbb{R}$ and all $\omega\in \Omega$,
\begin{align}\label{eqn:Fbound}
\lfloor x\rfloor \leq F_{\omega}(x)< \lfloor x\rfloor +2.
\end{align}
\end{lem}
\begin{proof}
Given $x\in\mathbb{R}$, let $x=\lfloor x \rfloor+a$, where $a\in [0,1)$ is the fractional part of $x$. By the degree 1 property (\ref{eqn:degree1}), $F_\omega(x)=F_\omega(a)+\lfloor x \rfloor$. Since $F_\omega$ is strictly increasing, $F_\omega(a)\geq F_\omega(0)$, and $F_\omega(0)\geq 0$ by the standard lift property (\ref{eqn:stdlift}), giving the left inequality. Similarly, $F_\omega(a)<F_\omega(1)=F_\omega(0)+1<2$, giving the right inequality.
\end{proof}

We also need the following measurability result (see Aliprantis and Border \cite[Theorem 18.19]{AliprantisBorder2006}) which uses the compactness of $\mathbb{S}^1$.

\begin{lem}\label{lem:measinfsup}
Let $(\Omega, \mathcal{F}, \mathbb{P})$ be a probability space. Suppose that $u:\Omega\times \mathbb{S}^1\to \mathbb{R}$  is measurable with respect to the first argument for each fixed $x\in \mathbb{S}^1$ and continuous with respect to the second argument for $\mathbb{P}$-almost every fixed $\omega\in\Omega$. Then the functions
\begin{align*}
\omega\mapsto \sup_{x\in \mathbb{S}^1} u(\omega,x),
\qquad
\omega\mapsto \inf_{x\in \mathbb{S}^1} u(\omega,x)
\end{align*}
are measurable.    
\end{lem} 

These lemmas can be used to obtain bounds for the mean random rotation number.

\begin{prp}\label{prp:rhobound}
Let $(\Omega, \mathcal{F}, \mathbb{P}, \sigma)$ be a measure-preserving dynamical system and let $F\in \tilde{\mathcal{H}}(\Omega)$ be the standard lift of $f$. Then $\rho(F)\in [0, 1]$. Moreover,
\begin{align}
\int_\Omega m\:\mathrm{d}\mathbb{P}\leq \rho(F) \leq \int_\Omega M\:\mathrm{d}\mathbb{P},
\end{align}
where
\begin{align}
m(\omega)=\inf_{x\in \mathbb{S}^1}\mathrm{dev}_F(\omega,x)\quad \mathrm{and} \quad M(\omega)=\sup_{x\in \mathbb{S}^1}\mathrm{dev}_F(\omega,x).
\end{align}
\end{prp}
\begin{proof} 
By Lemma \ref{lem:Fbound}, for $x\geq 0$ and $n\in\mathbb{N}$, we have $F^{(n)}_\omega(x)\geq 0$. Moreover, for $x\in [0, 1)$, we have $F_\omega(x)<2$,  and if $F^{(n)}_\omega(x)<n+1$ for some $n\in\mathbb{N}$, then it follows that $F^{(n+1)}_\omega(x)=F_{\sigma^n\omega}(F^{(n)}_\omega(x))<\lfloor F^{(n)}_\omega(x)\rfloor\leq n+2$. Hence, by induction, for all $n\in\mathbb{N}$ we have $0\leq F^{(n)}_\omega(x)<n+1$ and therefore
\begin{align}
0\leq \frac{F^{(n)}_\omega(0)}{n}<\frac{n+1}{n}
\end{align}
for all $n\in\mathbb{N}$. By (\ref{eqn:randomrotationnumber}), the sequence $(F^{(n)}_\omega(0)/n)$ converges to $\rho_F(\omega)$ for $\mathbb{P}$-almost every $\omega\in\Omega$, and so $\rho_F(\omega)\in [0,1]$ for $\mathbb{P}$-almost every $\omega\in\Omega$.

For the second statement, notice that for each $n\in\mathbb{N}$ and $x\in\mathbb{R}$ we have $F^{(n)}_\omega(x)-x =\sum_{i=0}^{n-1}\mathrm{dev}_F(\sigma^i\omega,f^{(i)}_\omega(x))$ by equation (\ref{eqn:devsum}), and so
\begin{align}\label{eqn:mMbounds}
\frac{1}{n}\sum^{n-1}_{i=0}m(\sigma^i \omega)\leq \frac{F^{(n)}_{\omega}(x)-x}{n}\leq \frac{1}{n}\sum^{n-1} _{i=0}M(\sigma^i\omega).
\end{align} 
By Lemma \ref{lem:measinfsup}, $m$ and $M$ are measurable. By equation (\ref{eqn:randomrotationnumber}), the central expression converges to $\rho_F(\omega)$ for $\mathbb{P}$-almost every $\omega\in\Omega$, and the result then follows by applying Birkhoff's ergodic theorem to each of the outer expressions and then integrating over $\Omega$.
\end{proof}

\section{Random periodic cycles}

We now define the concept of a random periodic cycle for a RCH. For $q\in\mathbb{N}$, let $Z_q$ denote the finite set $\{0,1,\ldots,q-1\}$ as before.

\begin{defn}\label{defn:rpc}
For $q\in\mathbb{N}$, a \emph{random periodic cycle} is a collection $\{a_j\}_{j\in Z_q}$ of measurable functions $a_j:\Omega\to\mathbb{S}^1$ together with a permutation $\tau:Z_q\to Z_q$ such that for $\mathbb{P}$-almost every $\omega\in\Omega$
\begin{align}\label{eqn:ajorder}
0\leq  a_0(\omega)< a_{1}(\omega)< \cdots<a_{q-1}(\omega)<1
\end{align}
and
\begin{align}\label{eqn:covariance}
f_\omega(a_j(\omega))=a_{\tau(j)}(\sigma\omega)
\end{align}
for each $j\in Z_q$. 
For $p\in Z_q$, the \emph{$(p,q)$-permutation} is the bijection $\tau:Z_q\to Z_q$ given by
\begin{align}
\tau(j)=j+p\ \mathrm{mod}\:q
\end{align}
for each $j\in Z_q$. A \emph{random $(p,q)$-periodic cycle} is a random periodic cycle for which the permutation $\tau$ is the $(p,q)$-permutation. The \emph{period} of a random periodic cycle is the order of the permutation $\tau$: the  minimal natural number $n$ for which $\tau^n=\mathrm{Id}$. Thus the period of a random $(p,q)$-periodic cycle is $q/\mathrm{gcd}(p,q)$.
\end{defn}

We inherit simple examples of random periodic cycles from the deterministic setting. For example, by taking $\Omega$ to be a singleton and $\sigma$ the identity, the generator reduces to a single homeomorphism, which we could take to be the rigid rational rotation $f_\omega(x)=x+p/q\ \mathrm{mod}\:1$, for some natural numbers $p$ and $q$ with $p<q$. So if we set $a_j(\omega)=j/q$, the collection $\{a_j\}_{j\in Z_q}$ is a random $(p,q)$-periodic cycle for the RCH $(\sigma,f)$. 

We now present a non-trivial example of a random periodic cycle, developed from an example of a random fixed point constructed by Newman \cite{Newman2016}.

\begin{ex}
Let $(\Omega,\mathcal{F},\mathbb{P},\sigma)$ be the two-sided $(p_0,p_1)$-Bernoulli shift on the set of symbols $Z_1=\{0,1\}$. 
We construct the generator $f\in \mathcal{H}(\Omega)$ of the RCH $(\sigma,f)$, by setting $f_\omega=g_{\omega_0}$, where
\begin{align}
g_0(x)=\begin{cases}
\frac12x+\frac12 \quad x\in \left[0,\frac14\right)  \\
\frac32x+\frac14 \quad x\in \left[\frac14,\frac12\right)  \\
\frac12x-\frac14 \quad x\in \left[\frac12,\frac34\right)  \\
\frac32x-1 \quad x\in \left[\frac34,1\right) 
\end{cases}
\end{align}
and $g_1(x)=g_0(x)+1/8\ \mathrm{mod}\:1$ (see Figure \ref{fig:g0g1}). 

\begin{figure}[htb]
\centering
\includegraphics[width=0.5\linewidth]{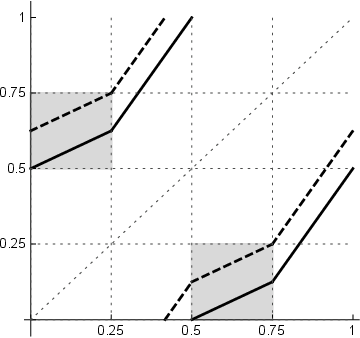}
\caption{Graphs of the circle homeomorphisms $g_0$ (solid line) and $g_1$ (dashed line).}
\label{fig:g0g1}
\end{figure}

For $j\in Z_2$ and $\omega\in\Omega$, let $a_j:\Omega\to\mathbb{S}^1$ be given by
\begin{align}
a_j(\omega) = \frac{j}{2}+\sum_{i=1}^\infty \frac{\omega_{-n}}{2^{n+2}}= (0.j0\omega_{-1}\omega_{-2}\omega_{-3}\ldots)_{\mathrm{base}\:2}.
\end{align}
Note that $a_0(\omega)\in [0,1/4]$ and $a_1(\omega)\in [1/2,3/4]$ for each $\omega\in\Omega$, so $0\leq a_0(\omega)<a_1(\omega)<1$. We have
\begin{align*}
f_\omega(a_0(\omega))
& =\begin{cases}
\frac12a_0(\omega)+\frac12, & \omega_0=0\\
\frac12a_0(\omega)+\frac58, & \omega_0=1
\end{cases}\\
& =\begin{cases}
(0.100\omega_{-1}\omega_{-2}\ldots)_{\mathrm{base}\:2}, & \omega_0=0\\
(0.101\omega_{-1}\omega_{-2}\ldots)_{\mathrm{base}\:2}, & \omega_0=1
\end{cases}\\
& = (0.10\omega_0\omega_{-1}\omega_{-2}\ldots)_{\mathrm{base}\:2}\\
& = a_1(\sigma\omega).
\end{align*}
Similarly, $f_\omega(a_1(\omega))=a_o(\sigma\omega)$. Therefore $f_\omega(a_j(\omega))=a_{j+1\:\mathrm{mod}\:2}(\sigma\omega)$ for each $j\in Z_2$, so the collection $\{a_j\}_{j\in Z_2}$ is a random $(1,2)$-periodic cycle.
\end{ex}

We show that random $(p,q)$-periodic cycles are the only random periodic cycles that are possible for a RCH.

\begin{prp}\label{prp:ranpqcycle}
If the RCH $(\sigma,f)$ has a random periodic cycle $\{a_j\}_{j\in Z_q}$ with permutation $\tau:\mathbb{Z}_q\to Z_q$ for some $q\in\mathbb{N}$, then $\{a_j\}_{j\in Z_q}$ is a random $(p,q)$-cycle, where $p=\tau(0)\in Z_q$.
\end{prp}
\begin{proof}
Let $r\in \mathcal{H}(\Omega)$ denote the rigid rotation $r_\omega(x)=x-a_p(\omega)\:\mathrm{mod}\:1$. Then $r_{\sigma\omega}(f_\omega(a_0(\omega)))=r_{\sigma\omega}(a_p(\sigma\omega))=0$ for $\mathbb{P}$-almost every $\omega\in\Omega$. Therefore, since $f_\omega$ and $r_{\sigma\omega}$ are orientation-preserving homeomorphisms, for $\mathbb{P}$-almost every $\omega\in\Omega$ we have 
\begin{align}\label{eq:rsigmaorder}
0=r_{\sigma\omega}(a_p(\sigma\omega))<r_{\sigma\omega}(a_{\tau(1)}(\sigma\omega))<\cdots < r_{\sigma\omega}(a_{\tau(q-1)}(\sigma\omega))<1.
\end{align}
Thus by composing with the rotation taking $a_p(\sigma\omega)$ to $0$, the cyclic order on the circle is converted into the linear order on $[0,1)$, thus forcing the permutation $\tau$ to be a rigid shift. By comparing indices in (\ref{eq:rsigmaorder}) with those in (\ref{eqn:ajorder}), we find that $\tau$ satisfies $j=\tau(j)-p\ \mathrm{mod}\:q$ for each $j\in Z_q$, and so $\tau$ is the $(p,q)$-permutation.
\end{proof}

Given a random periodic cycle $\{a_j\}_{j\in Z_q}$ of a RCH with permutation $\tau$, it follows from Proposition \ref{prp:ranpqcycle} that if $\tau$ fixes an element of $Z_q$, then $\tau=\mathrm{Id}$ and so $\{a_j\}_{j\in Z_q}$ is a random $(0,q)$-periodic cycle, or equivalently, $a_j:\Omega\to \mathbb{S}^1$ is a random fixed point for each $j\in Z_q$. If $p\geq 1$ and $p$ and $q$ are coprime, then the permutation $\tau$ is a single cycle and has order $q$. If $\mathrm{gcd}(p,q)=d\geq 1$, then the random $(p,q)$-cycle may be decomposed into $d$ distinct random $(p',q')$-cycles $\{b_{i,j}\}_{j\in Z_{q'}}$, $i\in Z_d$, where $p'=p/d$, $q'=q/d$ and
$b_{i,j}=a_{i+jd}$.

The existence of a random $(p,q)$-periodic cycle with minimal period at least two ensures that the random rotation number of the standard random lift is equal to $p/q$ for $\mathbb{P}$-almost every $\omega\in\Omega$.

\begin{thmA}
Let $F\in\tilde{\mathcal{H}}(\Omega)$ be the standard random lift of the generator $f\in \mathcal{H}(\Omega)$ of a random circle homeomorphism over a measure-preserving dynamical system $(\Omega,\mathcal{F},\mathbb{P},\sigma)$. If there is a random $(p,q)$-periodic cycle for $(\sigma, f)$ with $1\leq p<q$, then for $\mathbb{P}$-almost every $\omega\in\Omega$, the random rotation number $\rho_F\in L^1(\Omega)$ satisfies
\begin{align}
\rho_F(\omega) = \frac{p}{q}.
\end{align}
\end{thmA}
\begin{proof}
Given the random $(p,q)$-cycle, there is a full-measure set $\Omega'$ on which properties (\ref{eqn:ajorder}) and (\ref{eqn:covariance}) hold which we may assume is $\sigma$-invariant, else we can replace $\Omega'$ by $\bigcap_{n=0}^\infty \sigma^{-n}(\Omega')$. For each $j\in Z_q$ and for each $\omega\in\Omega'$, the fractional part of $F_\omega(a_j(\omega))$ is $a_{j+p\:\mathrm{mod}\:q}(\sigma\omega)$.
Since $F_\omega$ is strictly increasing, we have $F_\omega(a_{q-p-1}(\omega))<F_\omega(a_{q-p}(\omega))$, and so 
\begin{align*}
\lfloor F_\omega(a_{q-p-1}(\omega)) \rfloor + a_{q-1}(\sigma\omega) & < \lfloor F_\omega(a_{q-p}(\omega)) \rfloor + a_{q-p}(\sigma\omega) \\
& \leq  \lfloor F_\omega(a_{q-p}(\omega)) \rfloor + a_{q-1}(\sigma\omega),
\end{align*}
since $a_{q-p}(\sigma\omega)\leq a_{q-1}(\sigma\omega)$, from which we have 
\begin{align}\label{eqn:floorineq}
\lfloor F_\omega(a_{q-p-1}(\omega)) \rfloor < \lfloor F_\omega(a_{q-p}(\omega)) \rfloor.
\end{align}
For any $x\in [0,1)$, we have $0\leq F_\omega(x)<2$ by Lemma \ref{lem:Fbound}, and so $\lfloor F_\omega(x) \rfloor\in\{0,1\}$. Hence (\ref{eqn:floorineq}) implies
\begin{align}\label{eqn:jump}
\lfloor F_\omega(a_{q-p}(\omega)) \rfloor= \lfloor F_\omega(a_{q-p-1}(\omega)) \rfloor+1.
\end{align}

Since both $F_\omega$ and the floor function are non-decreasing, by (\ref{eqn:jump}) we have that $\lfloor F_{\omega}(a_j(\omega)) \rfloor$  is equal to $0$ if $j\leq q-p-1$ and is equal to $1$ if $j\geq q-p$: that is,
\begin{align}\label{eqn:chiq-p}
F_{\omega}(a_j(\omega))=a_{j+p\:\mathrm{mod}\:q}(\sigma\omega)+ \chi_{\{q-p,\ldots,q-1\}}(j).
\end{align}
In particular, in the case $j=0$, this means
\begin{align}
F_{\omega}(a_0(\omega))=a_{p\:\mathrm{mod}\:q}(\sigma\omega)+ \chi_{\{q-p,\ldots,q-1\}}(0).
\end{align}
Applying $F_{\sigma\omega}$ and using the degree 1 property, by (\ref{eqn:chiq-p}) we have
\begin{align*}
F^{(2)}_\omega(a_0(\omega)) 
&=F_{\sigma\omega}(a_{p\:\mathrm{mod}\:q}(\sigma\omega))+ \chi_{\{q-p,\ldots,q-1\}}(0)\\
&=a_{2p\:\mathrm{mod}\:q}+\chi_{\{q-p,\ldots,q-1\}}(0)+ \chi_{\{q-p,\ldots,q-1\}}(p\:\mathrm{mod}\:q)
\end{align*}
Inductively, for each $n\in\mathbb{N}$ we obtain
\begin{align*}
F_\omega^{(nq)}(a_0(\omega))-a_0(\sigma^{nq}\omega)
&=\sum_{i=0}^{nq-1} \chi_{\{q-p,\ldots,q-1\}}(ip\ \mathrm{mod}\:q)\\
&= \#\{0\leq i\leq nq-1: ip\ \mathrm{mod}\:q\in\{q-p,\ldots,q-1\}\}\\
&= \#\{1\leq i\leq nq: ip\ \mathrm{mod}\:q\in Z_p\}\\
&= n\cdot \#\{1\leq i\leq q: ip\ \mathrm{mod}\:q\in Z_p\}\\
&=np.
\end{align*}
To see the last step, let $d=\textrm{gcd}(p,q)$ and write $p=dp'$ and $q=dq'$. Note that for $1\leq i\leq q$, the quantity $ip\ \mathrm{mod}\:q$ is a multiple of $d$, of the form $jd$ with $j\in\{0,\ldots,q'-1\}$, with each multple of $d$ taken exactly $d$ times. Since the number of multiples of $d$ in $\{0,\ldots, p-1\}$ is $p'$, the cardinality of $\{1\leq i\leq q: ip\ \mathrm{mod}\:q\in Z_p\}$ is therefore $dp'=p$.

Therefore, since $a_0(\omega)-1<0\leq a_0(\omega)$ and $F^{(nq)}_\omega$ is strictly increasing, 
\begin{align*}
a_0(\sigma^{nq}\omega)+np-1<F^{(nq)}_\omega(0)\leq a_0(\sigma^{nq}\omega)+np.
\end{align*}
Thus by (\ref{eqn:randomrotationnumber}) we have that for $\mathbb{P}$-almost every $\omega\in\Omega$,
\begin{align}
\rho_F(\omega)=\lim_{n\to\infty} \frac{F^{(nq)}_\omega(0)}{nq}=\frac{p}{q}. 
\end{align}
\end{proof}

Note that we conclude that $\rho_F$ is essentially constant without assuming that the base dynamics is ergodic. An immediate consequence of Theorem A is a simple condition for the coexistence of random periodic cycles in a RCH.

\begin{cor}
If a RCH has a random $(p,q)$-periodic cycle and a random $(p',q')$-periodic cycle with $p\geq 1$ and $p'\geq 1$, then $p/q=p'/q'$.
\end{cor}
\begin{proof}
By Theorem A, we have that the mean rotation number is equal to both $p/q$ and $p'/q'$, giving a contradiction unless these values are equal.
\end{proof}

The requirement that $1\leq p<q$ is essential for Theorem A. The existence of a random fixed point does not ensure that the random rotation number of the standard random lift is zero. Indeed, the random rotation number can be non-zero even if the circle homeomorphisms $\{f_\omega\}_{\omega\in\Omega}$ all share a common fixed point, as the following example demonstrates.

\begin{ex}\label{ex:fixedpoint}
Let $(\Omega,\mathcal{F},\mathbb{P},\sigma)$ be the two-sided $(p_0,p_1)$-Bernoulli shift on the set of symbols $Z_1=\{0,1\}$. Let $(\sigma,f)$ be the random Arnold homeomorphism for which the generator $f\in \mathcal{H}(\Omega)$ is defined by equation (\ref{eqn:Arnold}) with the piecewise-constant functions
$\alpha:\Omega\to[-1,1]$ and $\beta:\Omega\to [0,1)$ given by 
\begin{align}
(\alpha(\omega),\beta(\omega)) = \begin{cases}
\left(-\frac{\pi}{5},+\frac{1}{10}\right),\quad\mathrm{if}\ \omega_0=0,\\
\left(+\frac{\pi}{5},+\frac{9}{10}\right),\quad\mathrm{if}\ \omega_0=1. \\
\end{cases}
\end{align}
This random Arnold homeomorphism has a common fixed point at $1/4$: for all $\omega\in\Omega$, $f_\omega(1/4)=1/4$ (see Figure \ref{fig:1/4}). The standard random lift $F\in \tilde{\mathcal{H}}(\Omega)$ of $f$ is given by
\begin{align}\label{eqn:Arnoldlift}
F_\omega(x) = x+\frac{\alpha(\omega)}{2\pi}\sin(2\pi x) + \beta(\omega).
\end{align}
So we have $F_\omega(1/4)=1/4$ if $\omega_0=0$ and $F_\omega(1/4)=5/4$ if $\omega_0=1$. Thus
\begin{align*}
\frac{1}{n}\left(F_\omega^{(n)}(1/4)-1/4\right) & = \frac{1}{n}\#\{i=0,\ldots, n-1: \omega_i=1\} \\
& = \frac{1}{n}\sum_{i=0}^{n-1} \chi_{C^0_1}(\sigma^i\omega)
\end{align*}
which, by Birkhoff's ergodic theorem, converges to $\mathbb{P}(C^0_1)=p_1$ for $\mathbb{P}$-almost every $\omega\in\Omega$. Thus $\rho_F(\omega)=p_1\in(0,1)$ for $\mathbb{P}$-almost every $\omega\in\Omega$ even though $f_\omega$ has a fixed point for each $\omega\in\Omega$.
\end{ex}

\begin{figure}[htb]
\centering
\includegraphics[width=0.5\linewidth]{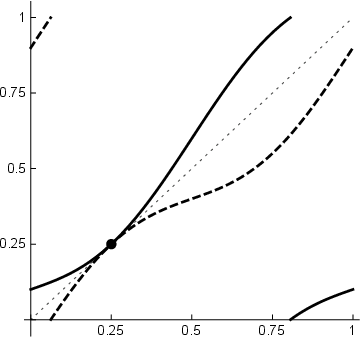}
\caption{Graphs of the fibre map $f_\omega$ when $\omega_0=0$ (solid line) and when $\omega_0=1$ (dashed line).}
\label{fig:1/4}
\end{figure}

The existence of a random fixed point $a_0:\Omega\to \mathbb{S}^1$ may not ensure that the random rotation number is zero, but it supports an invariant probability measure for the associated skew product that can be used to calculate the mean random rotation number. If we take $\mu_\omega$ to be the Dirac measure supported on $a_0(\omega)$ given by $\delta_{a_0(\omega)}(A)=\chi_A(a_0(\omega))$ for $A\in\mathcal{B}(\mathbb{S}^1)$ and $\omega\in\Omega$, then the probability measure $\mu\in\mathcal{P}_{\mathbb{P}}(\Omega\times\mathbb{S}^1)$ defined by equation (\ref{eqn:samplemeasure}) is an invariant measure for the associated skew product $S_f$, and so by equation (\ref{eqn:devint}) the mean random rotation number can be expressed as
\begin{align}
\rho(F)
=\int_{\omega\in\Omega}\int_{\mathbb{S}^1}\mathrm{dev}_F\:\mathrm{d}\delta_{a_0(\omega)}\:\mathrm{d}\mathbb{P}(\omega)
=\int_{\omega\in\Omega}\mathrm{dev}_F(\omega,a_0(\omega))\:\mathrm{d}\mathbb{P}(\omega).
\end{align}

\section{Integer-valued mean random rotation number}

In the deterministic setting, whenever the rotation number of the lift of a circle homeomorphism is an integer, then the homeomorphism has a fixed point. The following result shows that for a RCH, if the mean random rotation number of the standard random lift is an integer, then the homeomorphisms have a fixed point with positive probability.

\begin{thmB}
Let $F\in\tilde{\mathcal{H}}(\Omega)$ be the standard random lift of the generator $f\in \mathcal{H}(\Omega)$ of a random circle homeomorphism over a measure-preserving dynamical system $(\Omega,\mathcal{F},\mathbb{P},\sigma)$. If the mean random rotation number satisfies $\rho(F)\in \mathbb{Z}$, then 
\begin{align}
\mathbb{P}(\{\omega\in\Omega: f_\omega\ \mathrm{has\ a\ fixed\ point}\})>0.
\end{align}
\end{thmB}
\begin{proof}
Let $\Omega'=\{\omega\in \Omega: f \textrm{ has a fixed point}\}$ and suppose for a contradiction that $\mathbb{P}(\Omega')=0$. Then the complement $\Omega^0=\Omega\backslash \Omega'$ satisfies $\mathbb{P}(\Omega^0)=1$. Thus, for each $\omega \in \Omega^0$, the homeomorphism $f_\omega$ has no fixed point. As $F_\omega$ is the standard lift of $f_\omega$, which has no fixed points, we obtain $\mathrm{Dev}_F(\omega,0)\in (0,1)$. By continuity, we have $\mathrm{Dev}_F(\omega,x)\in (0,1)$ for all $x\in \mathbb{R}$: otherwise $\mathrm{Dev}_F(\omega,x)\in \mathbb{Z}$ for some $x\in\mathbb{R}$ by the intermediate value theorem, in which case $\pi(x)\in \mathbb{S}^1$ would be a fixed point of $f_\omega$, which is a contradiction.

Projecting the deviation to the circle, we thus have 
\begin{align}
0<\mathrm{dev}_F(\omega,x)<1
\end{align}
for all $x\in \mathbb{S}^1$ and all $\omega \in \Omega^0$. 
Since $x\mapsto\mathrm{dev}_F(\omega,x)$ is continuous, for each $\omega\in\Omega^0$ we have $m(\omega)=\inf_{x\in\mathbb{S}^1}\mathrm{dev}_F(\omega,x)>0$ and $M(\omega)=\sup_{x\in\mathbb{S}^1}\mathrm{dev}_F(\omega,x)<1$. By Lemma \ref{lem:measinfsup}, $m$ and $M$ are measurable.

Since $m>0$ on $\Omega^0$, we can write $\Omega^0$ as a union $\bigcup_{n\in \mathbb{N}} A_n$ of measurable sets, where $A_n=\{\omega \in \Omega^0: m(\omega)\geq \frac{1}{n}\}$. Since $\mathbb{P}(\Omega^0)=1$, there exist some $n\in\mathbb{N}$ for which $\mathbb{P}(A_n)>0$. Thus, by Proposition \ref{prp:rhobound}, we have
\begin{align}
\rho(F)
\geq \int_{\Omega} m\:\mathrm{d}\mathbb{P}
= \int_{\Omega^0} m\:\mathrm{d}\mathbb{P}
\geq \int_{A_n} m\:\mathrm{d}\mathbb{P}
\geq \int_{A_n} \frac{1}{n}\:\mathrm{d}\mathbb{P}
=\frac{\mathbb{P}(A_n)}{n}>0.
\end{align}
By similar reasoning, we have
\begin{align}
\rho(F)\leq \int_{\Omega^0}M\:\mathrm{d}\mathbb{P}<1, 
\end{align}
hence $\rho(F)\in (0,1)$, which contradicts with the assumption $\rho(F)\in \mathbb{Z}$ and therefore we conclude $\mathbb{P}(\Omega')>0$.
\end{proof}

\begin{rem}
Recall that for the standard random lift $F$, the mean random rotation number satisfies $\rho(F) \in [0,1]$ by Proposition \ref{prp:rhobound}. So the condition $\rho(F) \in \mathbb{Z}$ reduces to $\rho(F) \in \{0,1\}$.
\end{rem}

When the mean random rotation number is an integer, it is possible that the homeomorphisms also have no fixed point with positive probability, as the following example demonstrates.

\begin{ex}\label{ex:fixedpointhalf}
Let $\Omega$ be the two-point set $Z_1=\{0,1\}$ and consider $(\Omega,\mathcal{F},\mathbb{P},\sigma)$, where $\sigma$ is the transposition $(01)$ and $\mathbb{P}=(\delta_0+\delta_1)/2$ is the unique $\sigma$-invariant probability measure on $\Omega$. Let $(\sigma,f)$ be the random Arnold homeomorphism for which the generator $f\in \mathcal{H}(\Omega)$ is defined by equation (\ref{eqn:Arnold}) with the piecewise-constant functions $\alpha:\Omega\to[-1,1]$ and $\beta:\Omega\to [0,1)$ given by 
\begin{align}
(\alpha(\omega),\beta(\omega)) = \begin{cases}
\left(0,\frac{1}{10}\right),\quad\mathrm{if}\ \omega=0,\\
\left(-\frac{\pi}{5},0\right),\quad\mathrm{if}\ \omega=1. \\
\end{cases}
\end{align}
Note that for each $n\in\mathbb{N}$, we have $F^{(2n)}_1(0.15)=0.15$ and $F^{(2n)}_2(0.25)=0.25$ (see Figure \ref{fig:0.15}). So $\rho_F(\omega)=0$ for all $\omega\in\Omega$. Since $f_0$ is the rigid rotation by $1/10$, which has no fixed points, and $f_1$ has a fixed point at $0$, we have $\mathbb{P}(\{\omega\in \Omega: f_\omega \mathrm{\ has\ a\ fixed\ point}\})=1/2$.
\end{ex}

\begin{figure}[htb]
\centering
\includegraphics[width=0.5\linewidth]{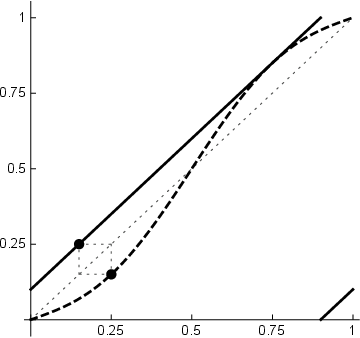}
\caption{Graphs of the fibre map $f_\omega$ when $\omega=1$ (solid line) and when $\omega=2$ (dashed line).}
\label{fig:0.15}
\end{figure}

It is straightforward to demonstrate no other values of the mean random rotation number ensure that periodic points exist with positive probability.

\begin{ex}\label{ex:periodic0}
Let $(\Omega,\mathcal{F},\mathbb{P},\sigma)$ be the two-sided $(\frac12,\frac12)$-Bernoulli shift on the set of symbols $Z_1=\{0,1\}$. Given $a\in (0,1)$, choose $b\in (0,\min\{a,1-a\})$ for which $a+b$ and $a-b$ are irrational.
We construct the generator $f\in \mathcal{H}(\Omega)$ of the RCH $(\sigma,f)$, by setting $f_\omega=g_{\omega_0}$, where
$g_0(x)=x+a-b\:\mathrm{mod}\:1$ and $g_1(x)=x+a+b\:\mathrm{mod}\:1$. Let $F\in\tilde{\mathcal{H}}(\Omega)$ denote the standard random lift of $f$. Then for each $n\in\mathbb{N}$, we have
\begin{align*}
\frac{1}{n}F_\omega^{(n)}(0) 
& = \frac{1}{n}\sum_{i=0}^{n-1} (a-b)\chi_{C^0_0}(\sigma^i\omega) + \frac{1}{n}\sum_{i=0}^{n-1} (a+b)\chi_{C^0_1}(\sigma^i\omega)\\
& = a-\frac{1}{n}\sum_{i=0}^{n-1} b\chi_{C^0_0}(\sigma^i\omega) + \frac{1}{n}\sum_{i=0}^{n-1} b\chi_{C^0_1}(\sigma^i\omega)
\end{align*}
which, by Birkhoff's ergodic theorem, converges for $\mathbb{P}$-almost every $\omega\in\Omega$ to $a$. However, as $a\pm b$ is irrational, we have that each map $f_\omega$ is an irrational rotation and so
\begin{align*}
\mathbb{P}(\{\omega\in\Omega: f_\omega\ \mathrm{has\ a\ periodic\ point}\})=0.
\end{align*}
\end{ex}

\end{document}